\documentclass[12pt,a4paper]{article}
\usepackage{bbm}
\usepackage{graphicx}
\graphicspath{ {./images/} }
\usepackage{wrapfig}
\usepackage[mathscr]{euscript}
\usepackage{subcaption}
\usepackage{hyperref}
\usepackage{epstopdf}
\usepackage{enumerate}
\usepackage{amsmath,amsfonts,amssymb,amsthm,epsfig,epstopdf,titling,url,array}
\usepackage{color}
\usepackage{framed}
\usepackage[utf8]{inputenc}
\usepackage[english]{babel}

 \usepackage{xparse}
 
 \NewDocumentCommand{\INTERVALINNARDS}{ m m }{
 	#1 {,} #2
 }
 \NewDocumentCommand{\interval}{ s m >{\SplitArgument{1}{,}}m m o }
 {
 	\IfBooleanTF{#1}{
 		\left#2 \INTERVALINNARDS #3 \right#4
 	}{
 		\IfValueTF{#5}{
 			#5{#2} \INTERVALINNARDS #3 #5{#4}
 		}{
 			#2 \INTERVALINNARDS #3 #4
 		}
 	}
 }
 \newtheorem{theorem}{Theorem}[section]
 \newtheorem{corollary}{Corollary}[section] 
 \newtheorem{definition}{Definition}[section]
 \newtheorem{lemma}[theorem]{Lemma}
 
 \newtheorem{prop}[theorem]{Proposition}
 
\newtheorem{Remark}{Remark}[section]

\usepackage{authblk}

\begin{document} 
\title{\textbf{The Julia sets of Chebyshev's method with small degrees}}
\author[1]{Tarakanta Nayak\footnote{Corresponding author: tnayak@iitbbs.ac.in}}
\author[1]{Soumen Pal\footnote{sp58@iitbbs.ac.in}}
\affil[1]{\textit{School of Basic Sciences\hspace{9cm}Indian Institute of Technology Bhubaneswar, India}}
\date{}
\maketitle
\begin{abstract}
 Given a polynomial $p$, the degree of its Chebyshev's method $C_p$ is determined. If $p$ is cubic then the degree of $C_p$ is found to be $4,6$ or $7$ and we investigate the dynamics of  $C_p$  in these cases.  If a cubic polynomial $p$ is unicritical or non-generic then, it is proved that the Julia set of $C_p$ is connected. The family of all rational maps arising as the Chebyshev's method  applied to a  cubic polynomial which is non-unicritical and generic is parametrized  by the multiplier of one of its extraneous fixed points. Denoting a member of this family with an extraneous fixed point with multiplier $\lambda $ by  $C_\lambda$, we have shown that the Julia set of $C_\lambda$ is connected  whenever $\lambda \in [-1,1]$. 
\end{abstract}
 \section{Introduction}
 Finding the roots of a given polynomial is a classical and widely studied topic. A root-finding method is a function that associates a polynomial $p$ with a rational map $F_p$ such that  each root $z$ of $p$ is an   attracting   fixed point  of $F_p$, i.e., $F_p (z)=z$ and $|F_p '(z)|<1$. It is well-known that there is an open connected subset of the extended complex plane $\widehat{\mathbb{C}}$ containing the attracting fixed point such that  every point in this set converges to the attracting fixed point under the iteration of $F_p$. This is how a root of a polynomial can be approximated starting with a suitably chosen point. 
 \par 
  The Fatou set of a rational map $R$, denoted by $\mathcal{F}(R)$ is the set of all points where $\{R^n\}_{n>0}$ is equicontinuous, and its complement in $\widehat{\mathbb{C}}$ is called the Julia set of $R$. The Julia set of $R$ is denoted by $\mathcal{J}(R)$. Complex dynamics is the study of the Fatou set and the Julia set of a given rational map. For an introduction to the subject, one may refer to a book by Beardon\cite{Beardon1991}.  The root-finding methods present themselves as an interesting class of  rational maps from a dynamical point of view. 
  \par   A fixed point $z_0 \in \mathbb{C}$ of a rational map $R$ is called attracting,  repelling or indifferent if the modulus of its multiplier $|R'(z_0)|$ is less than, greater than or is equal to $1$ respectively. The fixed point $z_0$ is called superattracting if  $R'(z_0)=0$.  If $\infty$ is a fixed point of $R$ then its multiplier is defined as $|h'(0)|$ where $h(z)=\frac{1}{R(\frac{1}{z})}$ and is called attracting, repelling or indifferent accordingly. The basin of attraction of an attracting fixed point $z_0$, denoted by $\mathcal{A}_{z_0}$ is the set $\{z \in \widehat{\mathbb{C}}: \lim_{n \to \infty} R^{n}(z)=z_0\}$. This is an open set and is not necessarily connected. The component of $\mathcal{A}_{z_0}$ containing $z_0$ is called the immediate attracting  basin of $z_0$. An indifferent fixed point is called rationally indifferent or parabolic if its multiplier is a root of unity. The basin of a parabolic fixed point $z_0$ of a rational map $R$ is the set $\{z \in \widehat{\mathbb{C}} \setminus \mathcal{J}(R): \lim_{n \to \infty} R^{n}(z)=z_0\}$. Every component of this basin contains $z_0$ on its boundary and is called an immediate parabolic basin. It is important to note that any point whose iterated image is $z_0$ is not in the basin of the parabolic fixed point $z_0$.  The basin of an attracting or a parabolic fixed point  is in the Fatou set. The immediate attracting basin or the parabolic basin corresponding to a periodic point of period $p$ can be defined accordingly and is a $p$-periodic Fatou component (maximally connected subset of the Fatou set which is invariant under $R^p$). The only other possible periodic Fatou component is a Siegel disk or a Herman ring. The details can be found in~\cite{Beardon1991}.
 
\par  One of the widely discussed root-finding methods is the Newton method and that is the first member of a family known as the Konig’s method. A systematic study of Konig's method is done by Buff and Henriksen in~\cite{BuffHenriksen2003}. A fixed point of a root-finding method $F_p$ is called extraneous if it is not a root of $p$. An important aspect of Konig's method is that all its extraneous fixed points are repelling. 
 This article is concerned with a root-finding method for which an extraneous fixed point can be non-repelling.

  \par 
   For a non-constant, non-linear polynomial $p$, its Chebyshev's method is defined as the rational map
$$
  C_{p}(z)=z-(1+\frac{1}{2}L_{p}(z))\frac{p(z)}{p'(z)},
  $$
  where
  $$
  L_{p}(z)=\frac{p(z)p''(z)}{p'(z)^{2}}.
 $$
  This is  a third order convergent method, i.e., its local degree  (made precise in the following paragraph) is three at each simple root of the polynomial $p$.
  Note that for a monomial or for a linear polynomial $p$, $C_p$   is a linear polynomial. We donot consider this trivial situation and are concerned with  polynomials with degree at least two and which are not monomials.
  \par
  The degree of a rational map is the first thing one needs to know for investigating its dynamics.
  While discussing a number of basic properties of $C_p$, the authors in \cite{GV2020} mention that the degree of $C_p$ is at most $3d-2$ where $d$ is the degree of the polynomial $p$. We found that the exact degree of $C_p$ depends not only  on the number of distinct roots of $p$ but on certain types of its critical points also. We determine the exact degree of $C_p$ for every $p$. Some discussion and definitions are required to state this result. 
  If  a rational map $R$ is analytic at a point $z_0$ and its Taylor series about $z_0$ is $a_k (z-z_0)^k +a_{k+1} (z-z_0)^{k+1}+\cdots $ for some $k > 0$ where $a_k \neq 0$ then we say the local degree of $R$ at $z_0$,  denoted by $\deg(R, z_0)$ is $k$. The map $R$ is \textit{like} $z \mapsto z^k$ near $z_0$. The local degree of $R$ at $\infty$ or at a pole is defined by a change of coordinate using $z \mapsto \frac{1}{z}$. More precisely, if $R(\infty)$ is finite then $\deg(R,\infty)$ is defined as the degree of $R(\frac{1}{z})$ at $0$. If $R(\infty)=\infty$ then $\deg(R,\infty)$ is  defined as $\deg(\frac{1}{R(\frac{1}{z})},0)$. Similarly, if $z_0$ is a pole of $R$ then its local degree at $z_0$ is defined to be  $\deg(\frac{1}{R(z)}, z_0)$. A root $\tilde{z}$ of a rational map $R$ is said to have multiplicity $k$ if $R(\tilde{z})=R'(\tilde{z})=\cdots R^{(k-1)}(\tilde{z})=0 $, but $R^{(k)}(\tilde{z}) \neq 0$ i.e., the local degree of $R$ at $\tilde{z}$ is $k$. 
  A root is called simple if its multiplicity is $1$. It is called multiple if its multiplicity is at least two. A root with multiplicity exactly equal to two is called a double root. A point $z \in \widehat{\mathbb{C}}$ is called  a critical point of a rational map $R$ if $\deg(R,z) \geq 2$. In particular, multiple roots and multiple poles  are critical points.  By definition, the multiplicity of a critical point $z$ of $R$ is $\deg(R,z)-1$. A critical point is called simple if its multiplicity is one.
   \begin{definition}(Special critical point)
  	For a polynomial $p$, a critical point $c\in \mathbb{C}$ is called  special if $p(c) \neq 0$ but $p''(c)=0$. 
  \end{definition}
  A  finite critical point with multiplicity at least $2$ and which is not a root is a special critical point.
  For example, $0$ is a special critical point of $p(z)=z^d+b$ whenever $d\geq 3$ and $b \neq 0$. But it is not so for $b=0$.
 We now present the first result of this article. 
   \begin{theorem}(Degree of $C_p$)
  	Let $p$ be a polynomial of degree $d$. Let $m, n$ and $r$ denote the number of its distinct simple roots, double roots and  roots of multiplicity bigger than $2$ respectively. If $p$ has $s$ number of distinct special critical points then  
  	$$\deg(C_p) =3(m+n+r)-2-B+s$$ where $B$ is the sum of multiplicities of all the special critical points.
  	If $p$ has no special critical point  then
  	$ 	\deg(C_p)=3(m+n+r)-2$.
  	\label{degree}
  \end{theorem}
	If $p$ is generic then $m=d$ and $n=r=0$, and
we have an immediate consequence.
\begin{corollary}
	If $p$ is generic then
	\begin{equation*}
	deg(C_p)=
	\begin{cases}3d-2-B+s~~\text{if $p$ has $s$ many special critical points with total multiplicy $B$}\\
	3d-2~~\text{if $p$ has no special critical point.} 
	\end{cases}
	\end{equation*}
\end{corollary}
The following corollary deals with some other special situations. 
\begin{corollary}
	\begin{enumerate}
		\item  If $p$  has two distinct roots then $\deg(C_p) = 4$.  In all other cases, $\deg(C_p) \geq 6$. In particular, there is no polynomial $p$ such that $\deg(C_p) =5$.
		\item If  $\deg(p)=d$ and all its critical points  are special then $\deg(C_p) =2d+s-1$ where $s$ is the number of distinct special critical points of $p$. Further, if $p$ is unicritical then $\deg(C_p) =2d$. 
	\end{enumerate}
\label{degree-corollary}
\end{corollary}
The degree of  the Chebyshev's method applied to a cubic polynomial is found to be $4,6$ or $7$. The remaining part of this article focusses on the dynamics of the Chebyshev's method in these cases. 
\par The dynamics of the Chebyshev's method for quadratic polynomials has been investigated by Kneisl  in~\cite{Kneisl2001} who calls this as Super-Newton method. The author gives examples of cubic polynomial whose Chebyshev's method has a superattracting extraneous fixed point. Olivo et al. \cite{Olivo2015}  consider a one parameter family of cubic polynomials and study their Chebyshev's method.  The case of cubic polynomials are also studied in \cite{GV2020} where the authors have discussed the cubic polynomials whose Chebyshev's method has attracting extraneous fixed points  and attracting periodic points.   In all these results, the connectedness of the Julia set remains unexplored. 
\par
 For a polynomial $p$, the  Chebyshev-Halley method of order $\sigma$ is given by
$
H^{\sigma}_p(z)=z-\left[1+\frac{1}{2}\dfrac{p(z)p''(z)}{(p'(z))^2-\sigma p(z)p''(z)}\right]\frac{p(z)}{p'(z)}
$
where $\sigma \in \mathbb{C}.$
 The Chebyshev's method is a special member of the  Chebyshev-Halley family.
In fact, $C_p =H_p ^0.$  The dynamics of the  Chebyshev-Halley family applied to unicritical polynomials $z \mapsto z^n -1, n \in \mathbb{N}$ is investigated in \cite{CCV2020}. A necessary and sufficient condition for disconnected Julia sets is found. More precisely, it is proved that the Julia set of these root-finding methods is disconnected if and only if the immediate basin of $1$ (which is a superattracting fixed point corresponding to a root of the polynomial) contains a critical point but no pre-image of $1$ other than itself. The numerical study done in this paper suggests the existence of disconnected Julia set for the Chebyshev-Halley method $H_{p}^{\sigma}$ for several values of $\sigma$ when $p$ is a cubic or higher degree unicritical polynomial. However, the paper does not contain any theoretical proof for this statement. 
\par  Though the Julia set of Newton method (applied to a polynomial) is always connected~\cite{Shishikura2009}, there are other members of Konig's methods with a disconnected Julia set~\cite{Honorato2013}. The Chebyshev's method applied to non-generic cubic polynomials is dealt in \cite{Olivo2015}, where   the connectivity question of their Julia sets remain to be answered.

\par We prove that the Julia set of  Chebyshev's method of each cubic polynomial that are either unicritical or non-generic is connected.  This is Proposition~\ref{lambda=5and6} of this article. In fact, this proposition shows that the Fatou set of $C_p$ is the union of the immediate superattracting basins (and their pre-images) corresponding to the three roots of $p$ when $p$ is unicritical whereas for non-generic $p$, the Fatou set of $C_p$ is the union of the two immediate attracting basins (and their pre-images) corresponding to the two roots of $p$.
\par 
As noted earlier,  the existence of an attracting or a rationally indifferent extraneous fixed point is a new feature of the Chebyshev's method. Its dynamical  relevance is revealed by parametrizing all the cubic polynomials in terms of the multiplier of an extraneous fixed point of its Chebyshev's method. This is done in Lemma \ref{characterization}. More precisely, it is shown that for every $\lambda \in \mathbb{C}\setminus \{5,6\}$, if the Chebyshev's method of a  cubic, generic and non-unicritical polynomial has an extraneous fixed point with multiplier $\lambda$ then it is conjugate to the Chebyshev's method $C_\lambda$ of $p_\lambda(z)=z^3+3z+\frac{3\lambda^2-39\lambda+124}{(5-\lambda)\sqrt{5-\lambda}}$ where the principal branch of the square root $\sqrt{5-\lambda}$ is considered. The case $\lambda=5$ is possible when $p$ is unicritical whereas there is no cubic polynomial which has an extraneous fixed point with multiplier equal to $6$ (Remark~\ref{5-6}(2)). Then we study the dynamics of $C_\lambda$ for $\lambda \in [-1,1]$ and prove the following.
 \begin{theorem}\label{connected_J_set}
 	For $-1 \leq \lambda \leq 1$, the Julia set of $C_\lambda$ is connected.
 	\label{connected}
 \end{theorem}
The proof of the theorem in fact describes the dynamics of $C_\lambda$ completely. The Fatou set of $C_\lambda$ is the union of the supreattracting immediate basins corresponding to the three roots of $p_\lambda$ and the immediate basin of the extraneous fixed point  with multiplier $\lambda$ - this is attracting if $\lambda \in (-1,1)$ and rationally indifferent if $\lambda=\pm 1$.
\par   
Though both $C_{-1}$ and $C_{1}$ have a rationally indifferent extraneous fixed point they differ in terms of the number of extraneous fixed points. The number of extraneous fixed point of $C_{-1}$ is four whereas it is three for $C_{1}$. A fixed point of a rational map $R$ is called multiple, with multiplicity $k \geq 2$ if it is a multiple root of $R(z)-z=0$ with multiplicity $k$. Otherwise, it is called simple. Here $C_1$ has a multiple fixed point whereas all the fixed points of $C_{-1}$ are simple. 
\par 
The images of the Julia set of $C_{\lambda}$ for $\lambda=-1,0$ and $1$ are given in Figure~\ref{Cheby-$-1$}, Figure~\ref{Cheby-$0$} and Figure~\ref{Cheby-$1$} respectively.  In each figure, the three immediate basins corresponding to the roots of the polynomial are given in blue, green and pink. The immediate basin of the extraneous attracting (for $C_{0}$) and parabolic (for $C_{-1}$ and $C_{1}$) fixed point  is indicated in red whose  zoomed version is also given in each figure.

\begin{figure}[h!]
	\begin{subfigure}{.5\textwidth}
		\centering
		\includegraphics[width=0.9\linewidth]{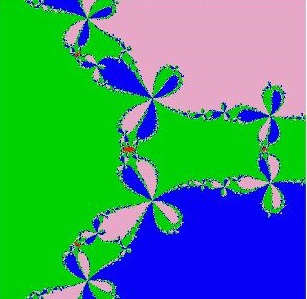}
		\caption{Periodic Fatou components}
	\end{subfigure}
	\begin{subfigure}{.5\textwidth}
		\centering
		\includegraphics[width=0.9\linewidth]{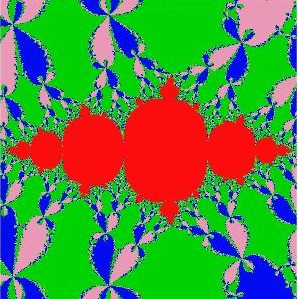}
		\caption{The immediate parabolic basin }
		
	\end{subfigure}
	\caption{The Julia set of $C_{-1} $ }
	\label{Cheby-$-1$}
\end{figure}
\begin{figure}[h!]
	\begin{subfigure}{.5\textwidth}
		\centering
		\includegraphics[width=1\linewidth]{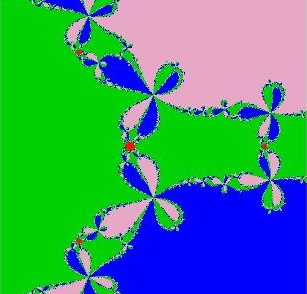}
		\caption{Periodic Fatou components}
	\end{subfigure}\hspace{0.05cm}
	\begin{subfigure}{.5\textwidth}
		\centering
		\includegraphics[width=1\linewidth]{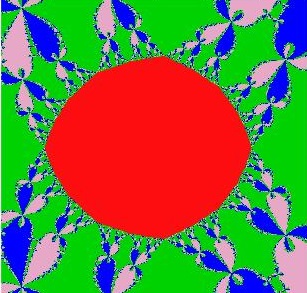}
		\caption{The immediate superattracting  basin}
	\end{subfigure}
	\caption{The Julia set of $C_0$}
	\label{Cheby-$0$}
\end{figure}
\begin{figure}[h!]
	\begin{subfigure}{.5\textwidth}
		\centering
		\includegraphics[width=0.9\linewidth]{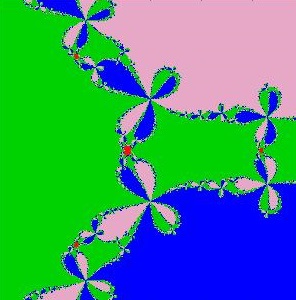}
		\caption{Periodic Fatou components}
	\end{subfigure}\hspace{0.0cm}
	\begin{subfigure}{.5\textwidth}
		\centering
		\includegraphics[width=0.9\linewidth]{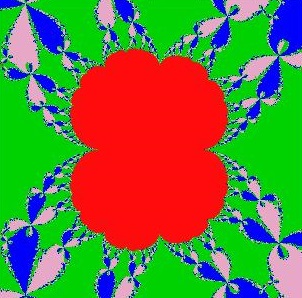}
		\caption{The immediate parabolic basin }
		
	\end{subfigure}
	\caption{The Julia set of $C_1 $ }
	\label{Cheby-$1$}
\end{figure}
\par 
Section 2 describes some useful properties of the Chebyshev's method. The degree of this method is determined in Section 3. The dynamics of $C_\lambda, -1 \leq \lambda \leq 1$ is investigated in Section 4.
 
 \section{Preliminaries}
 \subsection{Some properties of the Chebyshev's method }
Recall that, for a polynomial $p$,
 $$
C_{p}(z)=z-(1+\frac{1}{2}L_{p}(z))\frac{p(z)}{p'(z)},
$$
and  the derivative of the Chebyshev's method is given by
 \begin{equation}\label{deri}
 C_p'(z)=\frac{L_{p}(z)^2}{2}(3-L_{p'}(z))
 \end{equation}
 where  $ L_{p}(z)=\frac{p(z)p''(z)}{p'(z)^{2}}
$ and  $L_{p'}(z)=\frac{p'(z)p'''(z)}{p''(z)^2}.$
 
 Two rational maps $R,S$ are conformally conjugate, in short conjugate if there is a mobius map $\phi$ such that $S=\phi \circ R \circ \phi^{-1}$. Here $\circ$ denotes the composition of functions. Since $S^n=\phi \circ R^n \circ \phi^{-1}$ for all $n$, the iterative behaviour of $R$ and $S$ are essentially the same. More precisely, we have the following.
 \begin{lemma}[Theorem 3.1.4~\cite{Beardon1991}]
 If $S$ and $R$ are two rational maps such that $S=\phi \circ R \circ \phi^{-1}$ for a mobious map $\phi$ then $\mathcal{J}(S)=\phi (\mathcal{J}(R))$
 \end{lemma} 
 
There are different polynomials giving rise to the same Chebyshev's method up to conjugacy. The so called scaling theorem, which is also true for the Chebyshev-Halley method makes it precise. For brevity we use $H_p$ instead of $H_p ^{\sigma}$ for denoting the Chebyshev-Halley method of order $\sigma$ applied to $p$.
 \begin{theorem}[Scaling Theorem]\label{ST}
Let  $p$ be a polynomial of degree at least two. Then $H_p = H_{\lambda p} $ for all $\lambda \neq 0$. If $T (z) = \alpha z+ \beta$, with $\alpha, \beta \in \mathbb{C}, \alpha \neq 0$ and $g=  p\circ T$ then $T\circ H_g\circ T^{-1}=H_p.$
\end{theorem}
\begin{proof}
	For $g(z) = \lambda p(T(z))$, $g'(z)=\lambda p'(T(z))T'(z)=\lambda \alpha p'(T(z))$ and $g''(z)=\lambda \alpha^2 p''(T(z))$. Then
	\begin{align*}
	H_g(z) & = z-\left[1+\frac{1}{2}\dfrac{g(z)g''(z)}{ \left[g'(z) \right]^2 -\sigma g(z)g''(z)}\right]\frac{g(z)}{g'(z)}\\
	& = z-\left[1+\frac{1}{2}\dfrac{\lambda^2 \alpha^2 p(T(z))p''(T(z))}{ \left[\lambda \alpha p'(T(z))\right]^2-\sigma \lambda^2 \alpha^2 p(T(z))p''(T(z))} \right]\frac{\lambda p(T(z))}{\lambda \alpha p'(T(z))}\\
	& = z-\frac{1}{\alpha}\left[1+\frac{1}{2}\dfrac{p(T(z))p''(T(z))}{\left[ p'(T(z)) \right]^2-\sigma p(T(z))p''(T(z))}\right]\frac{p(T(z))}{p'(T(z))}.
	\end{align*}
	This implies that 
	\begin{align*}
	T\circ H_g(z)& =\alpha H_g(z)+\beta = T(z)-\left[1+\frac{1}{2}\dfrac{p(T(z))p''(T(z))}{[p'(T(z))]^2- \sigma p(T(z))p''(T(z))}\right]\frac{p(T(z))}{p'(T(z))}.
	\end{align*}
	This is nothing but $ H_p(T(z))$. Now putting $T(z)=z$ we get $H_p =H_{\lambda p}$. Similarly, for $\lambda =1$, we have $T \circ H_g \circ T^{-1}=T_p$.

\end{proof} 
\begin{Remark}
\begin{enumerate}
	\item For every polynomial $p(z)=a_d z^d+ a_{d-1}z^{d-1}+ \cdots  +a_1 z+a_0$, there is an affine map $T(z)=\alpha z+\beta$ where $\alpha^d =\frac{1}{a_d}$ and $\beta=-\frac{a_{d-1}}{d a_d}$  so that the coefficients of $z^d$ and $z^{d-1}$ in $(p \circ T)(z) $ are $1$ and $0$ respectively. It follows from the  Scaling theorem that $C_p$ is conjugate to $C_{p\circ T}$ leading to a considerable amount of simplification. We can assume without loss of generality that $p$ is monic or centered, or both as long as the dynamics of $C_p$ is concerned.
	\item In view of the previous remark, for a cubic polynomial $p$, we assume without loss of generality that $p(z)=z^3 +az+b$ for some $a,b \in \mathbb{C}$. Then $p'(z)=3z^2+a$, $p''(z)=6z$ and $p'''(z)=6$.
    In this case, 
    $$ L_p(z)=\frac{(z^3+az+b)6z}{(3z^2+a)^2}~~~\mbox{and} ~~~ L_{p'}(z)=\frac{3z^2+a}{6z^2}. 
    $$
    Hence
 $$C_p(z)
      =\frac{15z^7+6az^5-15bz^4-a^2z^3-12abz^2-3b^2z-a^2b}{(3z^2+a)^{3}},\label{cubic-Cheby}$$
    and $$ C' _p(z)= \frac{L_p (z)^2}{2}(3-L_{p'}(z))=\frac{3(z^3+az+b)^2 (15z^2-a)}{(3z^2+a)^4}.\label{cubic-Cheby-deri}$$
\label{cheby-for-cubic}
\end{enumerate}
\end{Remark}
As stated earlier, the roots of $p$ are fixed points of $C_p$. But every fixed points of $C_p$ is not necessarily a root of $p$. 

\begin{definition}
A fixed point of $C_p$ is called extraneous if it is not a root of $p$.  
\end{definition}

 The extraneous fixed points of $C_p$ can be attracting, repelling or indifferent. This is where the Chebyshev's method stands out from the comparatively well-studied Konig's methods, where all the extraneous fixed points are repelling.
Now we deal with  all the fixed points of $C_p$. Though the following is well-known, we choose to provide a proof  for the sake of completeness.
 \begin{prop}[Fixed points of $C_p$]
 	Let $C_p$ be the Chebyshev's method applied to a polynomial $p$ with degree $d$ where $d \geq 2$.
 	\begin{enumerate}
 		\item Every root of $p$ with multiplicity $k$  is a fixed point of $C_p$ with multiplier $\frac{(k-1)(2k-1)}{2k^2}$. In particular, every (simple) root of $p$ is an attracting (superattracting) fixed point of $C_p$.
 		\item  The point at $\infty$ is   a  fixed point of $C_p$ with multiplier  $\frac{2d^2}{2d^2-3d+1}.$ In particular, it is repelling.
 		\item A finite extraneous fixed point $\zeta$ of $C_p$ is precisely a root of $L_p (z)=-2$, and is attracting, repelling or indifferent if $2|3-L_{p'}(\zeta)|$ is less than, greater than or is equal to $1$ respectively.
 	\end{enumerate}
 \end{prop}
 \begin{proof}
 	\begin{enumerate}
 		\item Let $\alpha$ be a root of $p$ with multiplicity $k$. Then $p(z)=(z-\alpha)^kg(z)$ for some polynomial $g$ with $g(\alpha)\neq 0$. Further,
 	\begin{eqnarray}\nonumber
 	&p'(z)&=k(z-\alpha)^{k-1}g(z)+(z-\alpha)^kg'(z)
 =(z-\alpha)^{k-1}\left[kg(z)+(z-\alpha)g'(z)\right],
 	\end{eqnarray}
 	\begin{eqnarray}\nonumber
 	&p''(z)&=k(k-1)(z-\alpha)^{k-2}g(z)+2k(z-\alpha)^{k-1}g'(z)+(z-\alpha)^kg''(z)\\&& \nonumber
 	=(z-\alpha)^{k-2}\left[k(k-1)g(z)+2k(z-\alpha)g'(z)+(z-\alpha)^2g''(z)\right]
 	\end{eqnarray} and
 	\begin{eqnarray}\nonumber
 	&p'''(z)&=k(k-1)(k-2)(z-\alpha)^{k-3}g(z) +3k(k-1)(z-\alpha)^{k-2}g'(z)\\ &&\nonumber +3k(z-\alpha)^{k-1}g''(z)+(z-\alpha)^kg'''(z)\\&&
 	\nonumber
 	=(z-\alpha)^{k-3}\left[k(k-1)(k-2)g(z)\right.\\ &&\nonumber \left.+3k(k-1)(z-\alpha)g'(z)+3k(z-\alpha)^2g''(z)+(z-\alpha)^3g'''(z)\right].
 	\end{eqnarray}
 	Note that for $k=2$, the first term in the expression of $p'''$ vanishes and we have $p'''(z)=6g'(z)+6(z-\alpha)g''(z)+(z-\alpha)^2g'''(z).$ Hence we get
 	
 	$$	 L_p(z) =\frac{p(z)p''(z)}{(p'(z))^2} 
 	=\frac{g(z)\left[k(k-1)g(z)+2k(z-\alpha)g'(z)+(z-\alpha)^2g''(z)\right]}{\left[kg(z)+(z-\alpha)g'(z)\right]^2}.  $$
 
 	This gives
 \begin{equation}
  L_p(\alpha)=\frac{k-1}{k}.
  \label{Lp}
 \end{equation}  
Similarly it is found that $L_{p'}(\alpha)=\frac{k-2}{k-1}$. 	Hence
 $$ C_p'(\alpha) =\frac{[ L_p(\alpha) ]^2}{2}\left[3-L_{p'}(z)\right]    =\frac{(k-1)^2}{2k^2}\left[3-\frac{k-2}{k-1}\right]    =\frac{(k-1)(2k-1)}{2k^2}.
$$
The rest is straightforward.
 	
 	\item Since $C_{p}=C_{\lambda p}$ for each $\lambda \in \mathbb{C} \setminus \{0\}$ and every polynomial $p$,  without loss of generality we assume that $p$ is a monic polynomial. If $\deg(p)=d \geq 2$ then   $p(z)=z^d+a_1z^{d-1}+...+a_d$,
 	 $p'(z)=dz^{d-1}+(d-1)a_1z^{d-2}+...+a_{d-1}$  and
 	 $p''(z)=d(d-1)z^{d-2}+(d-1)(d-2)a_1z^{d-3}+...+2a_{d-2}.$ 
 	Now 
 	$$	 C_p(z) =z-\left(1+\frac{1}{2}\frac{p(z)p''(z)}{[p'(z)]^2}\right)\frac{p(z)}{p'(z)} 
 	=\frac{2z[p'(z)]^3-2p(z)[p'(z)]^2-[p(z)]^2p''(z)}{2[p'(z)]^3}. $$
 	
 	Here $2z[p'(z)]^3$, $2p(z)[p'(z)]^2$ and $[p(z)]^2p''(z)$ are all polynomials of the same degree $3d-2$ with the leading coefficients $2d^3, 2d^2$ and $ d(d-1) $ respectively. However $2(p'(z))^3$ is a polynomial with degree $3d-3$ and its leading coefficient is $2d^3$. Therefore
 	\begin{equation}\label{C_p}
 	C_p(z)=\frac{(2d^3-3d^2+d) z^{3d-2}+\alpha_{3d-1}z^{3d-1}+\dots +\alpha_{0}}{2d^3 z^{3d-3}+\beta_{3d-2}z^{3d-2}+\dots +\beta_{0}},
 	\end{equation}
for some $\alpha_0,\alpha_1,\alpha_2,\cdots, a_{3d-1}, \beta_0,\beta_1,\beta_2,\cdots, \beta_{3d-2} \in \mathbb{C}$.
 	Hence $C_p (\infty)=\infty$ and its multiplier is $ \frac{2d^3}{2d^3-3d^2+d}=\frac{2d^2}{2d^2-3d+1}$ (See page 41,\cite{Beardon1991}).
 	\item  Each solution of $L_p(z)=-2$ is a fixed point of $C_p$ but is not a root of $p$.  This is because the value of $L_p$ at each root of $p$ is in $(0,1)$ by Equation~(\ref{Lp}). Thus, the extraneous fixed points of $C_p$ are precisely the roots of $L_p(z)=-2$.  
 	It now follows from  Equation~(\ref{deri}) that the multipler of an extraneous fixed point $\zeta$ is $|2(3-L_{p'}(\zeta))|.$ The rest is obvious.
 	 	\end{enumerate}
 \end{proof}
\begin{Remark}
	\begin{enumerate}
		\item 
It is possible that the numerator and the denominator of $C_p$ in Equation~(\ref{C_p}) have a common factor making the  degree of $C_p$ strictly less than $2d^3-3d^2+d$. For example, if $p(z)=z^3+c,$ $c\neq 0$ then $deg(C_p)=6.$  (See Proposition~\ref{lambda=5and6} in Section 4)
In this case, the leading coefficients of the numerator and the denominator of $C_p$ changes after cancelling the common factors. However,  their ratio remains unchanged giving that the  multiplier of infinity is well-defined. 

\item Every fixed point of $C_p$ which is not attracting is extraneous. But an extraneous fixed point of $C_p$ can be attracting, repelling or indifferent depending on the nature of $p$.
\end{enumerate}
\label{fixedpoints}
\end{Remark}

  \section{Degree of the Chebyshev's method}
 
  A fixed point is multiple if and only if it is rationally indifferent with multiplier equal to $1$(See Page 142,~\cite{Milnor}). This fact is used in the following proof.

 \begin{proof}[Proof of Theorem \ref{degree}]
 	Let $p$ be a monic polynomial with simple root at $\alpha_i$, $i=1,2,\dots, m$; double root   at $\beta_j$, $j=1,2,\dots,n$ and root  $\gamma_k$, $k=1,2,\dots,r$ with multiplicity $a_k \geq 3.$ Then
 	$$p(z)=\prod_{i=1}^{m}(z-\alpha_i)\prod_{j=1}^{n}(z-\beta_j)^2\prod_{k=1}^{r}(z-\gamma_k)^{a_k}$$
 	and  $\deg(p)=d=m+2n+M$ where $M=\sum_{k=1}^{r}a_k.$
 
  If $C_p (z)=\frac{F(z)}{G(z)}$ then $\deg(F)=\deg(G)+1$ by Equation~(\ref{C_p}) and  therefore, the sum of all the roots of $F(z)-zG(z)=0$ counting multiplicities is nothing but $\deg(C_p)$. This is because the leading coefficients of $F$ and $G$ are different, $\infty$ is a simple fixed point of $C_p$ and the number of fixed points of $C_p$, counting multiplicity is $\deg(C_p)+1$.   Each   root of $p$ is an attracting or a superattracting fixed point of $C_p$, and these are simple roots of $F(z)-zG(z)=0$. Every other fixed point of $C_p$ are extraneous  and is a root of $L_p(z)+2=0$. As is evident, a multiple fixed point of $C_p$ with multiplicity $k$ is a multiple root of $L_p(z)+2=0$ with the same multiplicity and vice-versa. Thus,
 \begin{equation}
  \deg(C_p)=m+n+r+\deg(L_p).
  \label{degree-formula}
 \end{equation} 
   Now we need to find $\deg(L_p)$ in order to determine $\deg(C_p)$. 
  
  \par If $\alpha$ be a root of $p$ with multiplicity $k$, then it is a root of $p'$ with multiplicity $k-1$. Therefore
\begin{equation}
  p'(z)=\prod_{j=1}^{n}(z-\beta_j)\prod_{k=1}^{r}(z-\gamma_k)^{a_k-1}g(z) $$ and $$ p''(z)=\prod_{k=1}^{r}(z-\gamma_k)^{a_k-2}h(z)
  ,
  \label{g}\end{equation}
 	where $g$ and $h$ are some polynomials such that $g$ is non-zero at each $\alpha_i, \beta_j$ and $\gamma_k$ and $h$ is non-zero at each $\beta_j$ and $\gamma_k$, for $i=1,2,\dots,m,$ $j=1,2,\dots,n$, and $k=1,2,\dots,r.$ Here we donot rule out $h(\alpha_i)=0$ and that is possible, but not relevant here. Note that
 	\begin{align*}
 	\deg(g)& =\deg(p')-n-\sum_{k=1}^{r}(a_k-1)  =(d-1)-n-M+r =m+n+r-1
 	\end{align*}
 	and
 	\begin{align*}
 	\deg(h)& =\deg(p'')-\sum_{k=1}^{r}(a_k-2) =(d-2)-M+2r  =m+2n+2r-2.
 	\end{align*}
 	Now 
 	\begin{align*}
 	L_p(z)& =\dfrac{p(z)p''(z)}{[p'(z)]^2}\\
 	& =\dfrac{\prod_{i=1}^{m}(z-\alpha_i)\prod_{j=1}^{n}(z-\beta_j)^2\prod_{k=1}^{r}(z-\gamma_k)^{2a_k-2}h(z)}{\prod_{j=1}^{n}(z-\beta_j)^2\prod_{k=1}^{r}(z-\gamma_k)^{2a_k-2}[g(z)^2]} =\dfrac{\prod_{i=1}^{m}(z-\alpha_i)h(z)}{[g(z)]^2}
 	\end{align*}
 	Letting $L_p (z)=\dfrac{P(z)}{Q(z)}$, we note that every common root of $P$ and $Q$ is a root of $g$ and hence is different from each $\alpha_i, \beta_j$ and $\gamma_k$. Thus any such common root is not a root of $p$. Further, it is a common root of $g$ and $h$, i.e., it is a critical point as well as an inflection point of $p$ ($p''$ vanishes at this point). In other words, every common root of $P$ and $Q$, if exists, is a special critical point of $p$. Conversely, every special critical point is a common root of $P$ and $Q$ (in fact of $g$ and $h$).
 	
 \par 	
 	Let $p$ has $s$ number of distinct special critical points, say $c_j$, with multiplicity $b_j$ for $j=1,2,\dots, s$. Then $g(z)=\prod_{j=1}^{s}(z-c_j)^{b_j}\tilde{g}(z)$ and $h(z)=\prod_{j=1}^{s}(z-c_j)^{b_j-1}\tilde{h}(z)$ where  $\tilde{g}$ and $\tilde{h}$ are polynomials without any common root. In this case, 
 	\begin{align*}
 	L_p(z) 
 	=\dfrac{\prod_{i=1}^{m}(z-\alpha_i)\prod_{j=1}^{s}(z-c_j)^{b_j-1}\tilde{h}(z)}{\prod_{j=1}^{s}(z-c_j)^{2b_j}[\tilde{g}(z)]^2}  =
 	\dfrac{\prod_{i=1}^{m}(z-\alpha_i)\tilde{h}(z)}{\prod_{j=1}^{s}(z-c_j)^{b_j+1}[\tilde{g}(z)]^2}.
 	\end{align*}
 	Now	$ \deg(\tilde{g}) =\deg(g)-\sum_{j=1}^{s}b_j= m+n+r-1-B $ and $\deg(\tilde{h})=\deg(h)-\sum_{j=1}^{s}(b_j-1)=m+2n+2r-2-B+s.$ Therefore,
 $ \deg(  \prod_{i=1}^{m}(z-\alpha_i)\tilde{h}(z)   )=m+\deg(\tilde{h})=2m+2n+2r-2-B+s$
and  $\deg(\prod_{j=1}^{s}(z-c_j)^{b_j+1}\tilde{g}(z)^2 )  =\sum_{j=1}^{s}(b_j+1)+2\deg(\tilde{g})=B+s+2m+2n+2r-2-2B =2m+2n+2r-2-B+s.$ This implies that $\deg(L_p)= 2m+2n+2r-2-B+s.$ Hence by Equation~(\ref{degree-formula}), 
 	$$\deg(C_p)=3(m+n+r)-2-B+s.$$

 		If $p$ has no special critical point then $s=B=0$ and we get,
 	$$deg(C_p)=3(m+n+r)-2.$$
 
 \end{proof}

\begin{proof}[Proof of Corollary~\ref{degree-corollary}]
	\begin{enumerate}
		\item 	If $p$ has two distinct roots and $p(z)=(z-a)^m (z-b)^n$ for some $m,n \in \mathbb{N}, a,b \in \mathbb{C}$ then it has only one critical point different from the roots, namely $\frac{mb+na}{m+n}$. Further, it is a simple critical point. Hence $p$ has no special critical point and  $\deg(C_p)=4$. 
		\par We assert that four is the minimum possible degree of $C_p$ for every polynomial $p$. To see it, let $p$ have at least three distinct roots. Then $\deg(p) \geq 3$.  If $p$ has no special critical point then  it follows from Theorem~\ref{degree} that $\deg(C_p) \geq 7$. Now assume that $p$ has at least one special critical point. Since each special critical point of $p$ is a root of $g$  and $\deg(g) = (m+n+r)-1$ where $g,m,n,r $ are as given in the proof of Theorem~\ref{degree}, $s \geq 1$ and $B \leq m+n+r-1$.  Then  $\deg(C_p) \geq 3(m+n+r)-2-(m+n+r-1)+1=2 (m+n+r)$. Since  $ m+n+r \geq 3$, we have $\deg(C_p) \geq 6$. It is clear that $\deg(C_p)=5$ is never possible for any polynomial $p$.
\item 
Recall that a special critical point of a rational map is a critical point with multiplicity at least two which is not a root.
	If for a polynomial $p$ with degree $d$, all the critical points are special, then $p$ has no multiple roots, i.e., $p$ is generic. Clearly the number of roots of $p$ is $d$. Let
 $p'(z)=\prod_{j=1}^{s}(z-c_j)^{b_j}$ where $b_j\geq 2$ and  $p(c_j)\neq 0$ for any $j=1,2,\dots,s$. Then $\deg(p')=d-1=\sum_{j=1}^{s} b_j$ and
	 $p''(z)=\prod_{j=1}^{s}(z-c_j)^{b_j-1}q(z)$ 
	where $q(c_j)\neq 0$ for any $j=1,2,\dots,s$. So $\deg(p'')=d-2=\sum_{j=1}^{s}(b_j-1)+\deg(q)$. This implies that $\deg(q)=d-2-\sum_{j=1}^s (b_j -1)=s-1.$ Thus
	 $\deg(L_p(z)) =\deg( \frac{p(z)q(z)}{\prod_{j=1}^{s}(z-c_j)^{b_j+1}})=d+s-1.$ Now it follows from Equation~(\ref{degree-formula}) that $\deg(C_p)=2d+s-1$. Further, if $p$ is unicritical then $p(z)=(z-a)^d+b$, for some $a, b \in \mathbb{C}, b \neq 0$. If $d=2$ then by the previous part of this corollary, $\deg(C_p)=4$. If $d \geq 3$ then its only critical point is $a$ and that is special. Now, it follows from the preceeding lines that $\deg(C_p)=2d$.
	 	\end{enumerate}
\end{proof}
\section{ Dynamics  of $C_\lambda$ }
%
We need the following well-known results for the proofs.
 For a rational map $R$, let $C_R$ denote the set of all critical points of $R$.
\begin{lemma}
	Let $U$ be a periodic Fatou component of  a rational map $R$.
	\begin{enumerate}
		\item 	If $U$ is an immediate attracting basin or an immediate parabolic basin  then  $U \cap C_R \neq \emptyset$.
		\item 	If $U$ is a Siegel disk or a Herman ring  then its boundary is contained in the closure of  $\{R^n(c):n \geq 0~\mbox{and}~c \in C_R \}$.
	\end{enumerate}
	\label{basic-dynamics-lemma}
\end{lemma}
\begin{lemma}[Riemann-Hurwitz formula]
If $R: U \to V$	is a rational map between two of its Fatou components $U$ and $V$ then it is a proper map of some degree $d$ and $c(U)-2=d(c(V)-2)+C$ where $c(.)$ denotes the connectivity of a domain and $C$ is the number of critical points of $R$ in $U$ counting multiplicity. Further, if $c(V)=1$ and there is no critical point of $R$ in $U$ then $c(U)=1$.
	 \label{r-h}
\end{lemma}
The following lemma is crucial to prove the simple connectivity of Chebyshev's method applied to polynomials.
 \begin{lemma}
 	Let $R$ be a rational map for which $\infty$ is a repelling fixed point. If $\mathcal{A}$
is an unbounded invariant immediate basin of attraction  then its boundary contains at least one pole of $R$.  Further, if all the poles of $R$ are on the boundary of $\mathcal{A}$ and $\mathcal{A}$ is simply connected then the Julia set of $R$ is connected. 
\label{poleinboundary} \end{lemma}
\begin{proof}
 Let $s>0$ and $B_s=\{z: \sigma(z, \infty)<s\}$ where $\sigma$ denotes the spherical metric in $\widehat{\mathbb{C}}$. Choose a sufficiently small $s$  such that $B_s$ does not contain any critical value of $R$. This is possible as $\infty$ is a repelling fixed point of $R$. Then the set  $R^{-1}(B_s)$ has $d=\deg(R)$ components one of which, say $N_{0}$ contains   $\infty$. Further, $R$ is one-one on $N_0$. Let all other components of $R^{-1}(B_s)$ be denoted by $N_i, 1\leq i \leq d-1$. Let $w \in (B_s \cap \mathcal{A}) \setminus \{\infty\}$. 
 As the degree of $R: \mathcal{A}  \to \mathcal{A}  $ is at least two (as $\mathcal{A}$ contains at least one critical point by Lemma~\ref{basic-dynamics-lemma}), at least two pre-images of $w$ are in $\mathcal{A} $. Since $R$ is  one-one in $N_{0}$, there is a pre-image of $w$ in $\mathcal{A}  \cap  N_j$ for some $j, 1 \leq j \leq d-1$. This is true for all $s' < s$ and for each $w \in (B_{s'} \cap \mathcal{A} ) \setminus \{\infty\}$. By considering a sequence $s_n \to 0 $ and $w_n \in (B_{s_n} \cap \mathcal{A} ) \setminus \{\infty\}$ so that $w_n$s are distinct and $w_n \to \infty$, we get a sequence $z_n $ in $ \cup_{ 1\leq i \leq d-1} \mathcal{A}  \cap  N_i$. Since   $z_m \neq z_n$ for all $m \neq n$, there is a subsequence $z_{n_k}$ and $j^* \in \{1,2,\cdots, d-1\}$ such that $z_{n_k} \in \mathcal{A} \cap N_j^*$ for all $k$. This subsequence has a limit point and that cannot be anything but a pole of $R$. This pole is clearly in the Julia set of $R$ and thus on the boundary of $\mathcal{A}$.
 \par 
  If all the poles of $R$ are on the boundary of $\mathcal{A}$ and $\mathcal{A}$ is simply connected then the unbounded  component of the Julia set contains all the poles of $R$. Let $U$ be a  multiply connected Fatou component of $R$. Consider a Jordan curve $\gamma$ in $U$ that surrounds a point of the Julia set i.e., the bounded component of $\widehat{\mathbb{C}} \setminus \gamma$ intersects the Julia set. As $\infty \in \mathcal{J}(R)$ and the backward orbit of $\infty$ is dense in the Julia set, there is a point $z$ surrounded by $\gamma$ such that $R^k (z)$ is a pole of $R$. Without loss of generality, assume that $k$ is the smallest natural number such that $R^k (z)$ is a pole. Then the curve $R^{k}(\gamma)$ surrounds a pole of $R$ by the Open mapping theorem. The set $R^k (\gamma)$ is completely  contained in the Fatou set  whereas there is a Julia component containing $\infty$ and all the poles of $R$. This is not possible  proving that all the Fatou components are simply connected. In other words, the Julia set of $R$ is connected. 
\end{proof}

\begin{Remark}
	It can follow  from the arguments used in the above proof that even if $\mathcal{A}$ is not simply connected, all the  Fatou components other that itself  are simply connected whenever the boundary of $\mathcal{A}$
contains all the poles of $R$.	
\end{Remark}
	
\begin{prop}
	Let $p$ be a cubic polynomial.
	\begin{enumerate}
		\item If $p$ is unicritical then its Chebyshev's method is conjugate to $\frac{5z^6+5z^3-1}{9z^5}$ and its Julia set is connected.
		\item If $p$ is not generic  then its Chebyshev's method is conjugate to $\frac{5z^4+15z^3+24z^2+22z+6}{9(z+1)^3}$ and its Julia set is connected.
	\end{enumerate}
\label{lambda=5and6}
\end{prop} 
\begin{proof}
	 \begin{enumerate}
	 	\item Let $p(z)=(z-\alpha)^3+\beta$ for some $\alpha, \beta \in \mathbb{C}, \beta\neq 0$. 	If $\beta =r e^{i \theta}$ for $r>0$ and $\theta \in [0, 2 \pi )$ then $(p\circ T)(z)=-re^{i \theta} ( z^3 -1)$ where $T(z)= -r^{\frac{1}{3}} e^{i\frac{\theta}{3}}z+\alpha$. In view of the Scaling theorem, we assume without loss of any generality that  $p(z)=z^3-1$. Its Chebyshev's method is
	$$C_p(z)=\frac{5z^6+5z^3-1}{9z^5}~\mbox{and }~
	C_p'(z)=\frac{5(z^3-1)^2}{9z^6}.$$

Note that a rational map with degree $d$ has $2d-2$ critical points counting multiplicity.
	As $\deg(C_p)=6$, there are ten critical points counted with multiplicities and those are the three roots of $p$  each with multiplicity two and the pole $0$ with multiplicity four. 

	 As $\infty$ is a repelling fixed point, it is in the Julia set of $C_p$ giving that  $0\in \mathcal{J}(C_p).$ 
	Hence none of the  superattracting immediate basins contains any critical point other than the superattracting fixed point. Hence each immediate basin is simply connected by Theorem 3.9 ~\cite{Milnor}. These are the only periodic Fatou components by Lemma~\ref{basic-dynamics-lemma}. Every Fatou component different from these are simply connected by the Riemann-Hurwitz formula (Lemma~\ref{r-h}).
	\item Let $p(z)=(z-a)^2(z-b)$ for some $a,b \in \mathbb{C}$ and $a\neq b$. Then for the affine map $T(z)=\frac{a-b}{3}z+\frac{2a+b}{3}$, $\left(\frac{3}{a-b}\right)^3p(T(z))=(z-1)^2(z+2).$ In view of the Scaling theorem, we  assume without loss of generality that $p(z)=(z-1)^2(z+2).$ Then
	$$C_p(z)=\frac{5z^4+15z^3+24z^2+22z+6}{9(z+1)^3}$$ and  $C_p'(z)=\frac{(z+2)^2(5z^2+1)}{9(z+1)^4}.$ The critical points are $ -1, -2$ and $\pm \frac{i}{\sqrt{5}}$. Also $-2$ is a superattracting fixed point of $C_p$ whereas $1$ is an attracting fixed point.  Let $\mathcal{A}_{-2}$ and $\mathcal{A}_1$ be the immediate basins of attraction of $-2$ and $1$ respectively. Note that $-1 \in \mathcal{J}(C_p)$.  Since $\mathcal{A}_1$ must contain a critical point of $C_p$, it is either $\frac{i}{\sqrt{5}}$ or $-\frac{i}{\sqrt{5}}$.
Since all the coefficients of $C_p$ are real, $C_{p}^n (\overline{z})=\overline{C_{p}^n (z)}$ for all $n$ and $z \in \mathbb{C}$. This gives that each Fatou component intersecting the real line is symmetric with respect to $\mathbb{R}$. In particular, $\mathcal{A}_1$ is symmetric with respect to $\mathbb{R}$. This gives  that $\mathcal{A}_1 $ contains both these critical points $\pm\frac{i}{\sqrt{5}}$. Thus the only periodic Fatou components of $C_\lambda$ are $\mathcal{A}_1$ and $\mathcal{A}_{-2}$, by Lemma~\ref{basic-dynamics-lemma}.
	\par  If  $C_p(x)< x$ for any $x<-2$, then strictly increasingness of $C_p$ in $(-\infty,-2)$ will give that $\lim_{n \to \infty}C_p ^n (x) = \infty$, which is not possible as $\infty$ is a repelling fixed point. Therefore, $C_p(x)> x$ for all $x<-2$ and $\lim_{n \to \infty}C_p ^n (x) = -2$. In other words, $(-\infty, 2) \subset \mathcal{A}_{-2}$ showing that $\mathcal{A}_{-2}$ is unbounded. By Lemma~\ref{poleinboundary}, there is a pole of $C_p$ on the boundary of $\mathcal{A}_{-2} $. But $C_p$ has only one pole.  As $\mathcal{A}_{-2}$ does not contain any critical point other than $-2$, it is simply connected (Theorem 3.9, \cite{Milnor}). Now it follows from Lemma~\ref{poleinboundary}
 that the Julia set of $C_p$ is connected.

	 \end{enumerate}
	\begin{figure}[h!]
		\begin{subfigure}{.5\textwidth}
			\centering
			\includegraphics[width=0.9\linewidth]{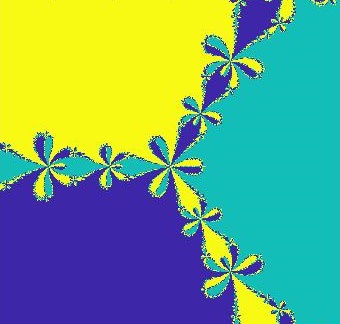}
			\caption{Unicritical: $p(z)=z^3-1$}
		\end{subfigure}
		\begin{subfigure}{.5\textwidth}
			\centering
			\includegraphics[width=0.9\linewidth]{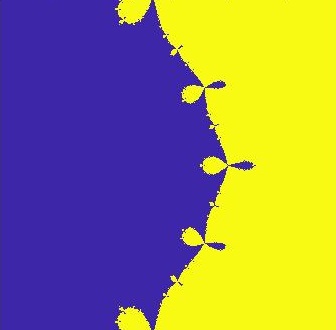}
			\caption{Non-generic: $p(z)=(z-1)^2(z+2)$}
		\end{subfigure}
		\caption{The Julia set of $C_p$ for unicritical and non-generic $p$ }
		\label{unicritical-nongeneric}
	\end{figure}
\end{proof}
The Fatou set of the unicritical polynomial $z^3 -1$ is given in Figure~\ref{unicritical-nongeneric}(a). The three superattracting basins are shown in blue, yellow and green.  The two basins of the roots of the non-generic polynomial $(z-1)^2(z+2)$ are shown in yellow and blue in Figure~\ref{unicritical-nongeneric}(b).
\begin{Remark} The Chebyshev's method of each unicritical cubic polynomial has three finite extraneous fixed points, each with multiplier $5$. For each non-generic cubic polynomial $p$, the finite extraneous fixed points of $C_p$ are with multipliers $9$ and $\frac{49}{9}$. This is because the multipliers of fixed points remain unchanged under conformal conjugacy.
\end{Remark}
Note that $\infty$ is always an extraneous fixed point of $C_p$ and its multiplier is $\frac{2 d^2}{2 d^2 -3d+1}$ where $d=\deg(C_p)$.
In order to deal with all non-unicritical and generic cubic polynomials, we use a parameterization in terms of the multiplier of a finite extraneous fixed point of the Chebyshev's method. The forbidden value  $5$  in the following lemma corresponds precisely to unicritical  polynomials whereas there is no cubic polynomial whose Chebyshev's method has a finite extraneous fixed point with multiplier $6$. In fact, the multiplier of $\infty$ can also never be $6$.
	\begin{lemma}\label{characterization}
	For  $\lambda \in \mathbb{C}- \{5,6\}$, if $p$ is a  non-unicritical and generic cubic polynomial  whose Chebyshev's method $C_p$ has a finite extraneous fixed point with multiplier $\lambda$ then $C_p$ is conjugate to the Chebyshev's method of $p_{\lambda}(z)=z^3+3z+\psi(\lambda)$, where $\psi(\lambda)=\frac{3\lambda^2-39\lambda+124}{(5-\lambda)\sqrt{5-\lambda}}$ and $\sqrt{5-\lambda}$ denotes the principal branch.
\end{lemma}
\begin{proof}
Let $p$ be a non-unicritical and generic cubic polynomial. Then $p(z)=z^3+az+b$ where $ a \neq 0, b \in \mathbb{C}$  and all the roots of $p$ are simple. Further, 
	$ L_p(z)=\frac{6z(z^3+az+b)}{(3z^2+a)^2}~\mbox{and}~ L_{p'}(z)= \frac{3z^2+a}{6z^2}$. 
	If $z$ is a finite extraneous fixed point of $C_p$ with multiplier $\lambda$ then, in view of Equation~(\ref{cubic-Cheby-deri}),
	\begin{equation}
	5-\frac{a}{3z^2}=\lambda \label{e1}.
	\end{equation} 
Note that the above equation has no finite solution for $\lambda=5$.  Now  either $\sqrt{\frac{a}{3(5-\lambda)}}$ or $-\sqrt{\frac{a}{3(5-\lambda)}}$ is the extraneous fixed point of $C_p$. Recall that, a fixed point $z$ of $C_p$ is extraneous if and only if $p(z)\neq 0$ and 
	$L_p(z)=-2\label{i}$. Since $p$ is cubic, generic and non-unicritical, neither $p$ and $p'$ nor $p'$ and $p''$ have any common root. A finite extraneous fixed point of $C_p$ is a solution of 
	\begin{equation}
	 \frac{p(z)p''(z)}{[p'(z)]^2}=-2,
	\label{LP}\end{equation} and any such solution is neither a root of $p$ nor a root of $p'$. Equation~(\ref{LP}) becomes
	\begin{equation}
	12 z^4+9a z^2+3b z +a^2=0.\label{iii}
	\end{equation}
For $\lambda =6$, the point $ \pm \sqrt{\frac{a}{3(5-\lambda)}}=\pm i \sqrt{\frac{a}{3}} $ becomes a root of $p'$ and hence has been avoided. 
\par 
Considering $-\sqrt{\frac{a}{3(5-\lambda)}}$ to be an extraneous fixed point, from Equation (\ref{iii}), we have
	$b=\frac{a\sqrt{a}}{3\sqrt{3}}\psi(\lambda)$. Similarly assuming that 
 $\sqrt{\frac{a}{3(5-\lambda)}}$ is the extraneous fixed point we have $b=-\frac{a\sqrt{a}}{3\sqrt{3}}\psi(\lambda)$.
	
	Let $p_1(z)=z^3+az+\frac{a\sqrt{a}}{3\sqrt{3}}\psi(\lambda)$ and $p_2(z)=z^3+az-\frac{a\sqrt{a}}{3\sqrt{3}}\psi(\lambda)$. Then $p_1(z)=-p_2(-z)$ and by the Scaling Theorem, $C_{p_1}$ and $C_{p_2}$ are conformally (in fact, affine) conjugate.
	Now, for $\phi(z)=\frac{\sqrt{a}}{\sqrt{3}}z$, $p_1(\phi(z))=\frac{a\sqrt{a}}{3\sqrt{3}}(z^3+3z+\psi(\lambda))$. Again applying the Scaling Theorem, we conclude that the Chebyshev's methods applied to $p_1$ and $z^3+3z+\psi(\lambda)$ are conjugate.
\end{proof}
\begin{Remark}
	\begin{enumerate}
		
			\item The extraneous fixed point of $C_\lambda, \lambda \neq 5,6$ having its multiplier equal to $\lambda$ is $-\frac{1}{\sqrt{5-\lambda}}$. 
		\item If $\lambda =5$  then Equation (\ref{e1}) gives that  $a=0 $ and $p(z)=z^3+b$ becomes an unicritical polynomial.
		\item For $\lambda =6$, the point that qualifies to be a finite extraneous fixed point of $C_p$ is $\pm i \sqrt{\frac{a}{3}}$. But $p' $ is zero  whereas $p''$ is not zero at this point. It cannot be a solution of Equation~(\ref{LP}).  This gives that there is no extraneous fixed point of $C_p$ with multiplier equal to $6$ for any non-unicritical and generic polynomial $p$. As seen in Proposition~\ref{lambda=5and6} and the remark following it, there is also no unicritical or non-generic cubic polynomial with an extraneous fixed point with multiplier equal to $6$. 
	\item  Since $\psi'(\lambda)=-\frac{3}{2} \frac{(6-\lambda)(1-\lambda)}{(5-\lambda)^2 \sqrt{5-\lambda}}$, $\psi'(\lambda) <0$ for all $\lambda  < 1$ and is $>0$ for all $1< \lambda <5$.	The function $\psi: (-\infty, 5) \to  [11, \infty)$ is  strictly decreasing in $(-\infty, 1)$, attains its minimum at $1$ and then it strictly increases. The minimum value is $\psi(1)=11$. See Figure~\ref{psi-lambda} for the graph of $\psi$.
\item \label{lambda<5}For all $\lambda  < 5$, $\psi(\lambda)$ is a real number and $p_\lambda(z)=z^3+3z+\psi(\lambda)$ preserves the real axis. In fact, all the coefficients in the numerator and the denominator of $C_\lambda$ are real and therefore $C_{\lambda}(\overline{z})=\overline{C_{\lambda}(z)}$ for all $z$. This gives that $C_{\lambda}^n(\overline{z})=\overline{C_{\lambda}^n(z)}$ for all $n$. In other words, the Fatou set of $C_\lambda$ is symmetric about the real line. If a Fatou component of $C_\lambda$ intersects the real line then it is also symmetric about the real line. 
	\end{enumerate}
\label{5-6}
\end{Remark}
	\begin{figure}[h!]
	\centering
	\includegraphics[width=0.6\linewidth]{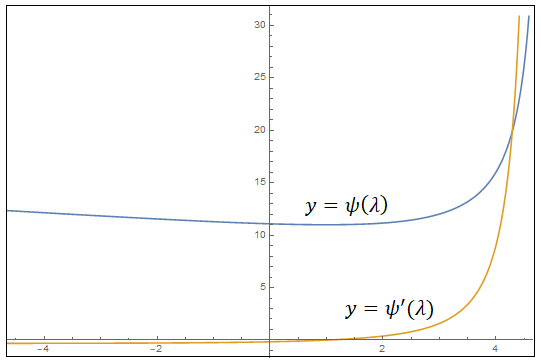}
	\caption{Graph of $\psi$}
	\label{psi-lambda}
\end{figure}
Now onwards, we consider $p_\lambda (z)=z^3 +3z+\psi(\lambda)$ and let  $C_{p_\lambda}$ be denoted by $C_\lambda$ for $\lambda \in \mathbb{C} \setminus \{5,6\}$. Then $L_{p_\lambda}(z)=\frac{p_\lambda (z) 2z}{3(z^2 +1)^2}$, 
 \begin{equation}
C_\lambda (z)=\frac{5z^7 +6z^5-5\psi(\lambda)z^4-3z^3-12 \psi(\lambda) z^2 -\psi(\lambda)^2 z-3 \psi(\lambda)}{9 (z^2 +1)^3},
\label{cheby-parameter} \end{equation}
 and
 \begin{equation}
 C_{\lambda}' (z)=\frac{[ p_{\lambda}(z) ]^2 (5z^2 -1)}{9 (z^2 +1)^4}.\label{cheby-prime-parameter}
    \end{equation}

    The following are some consequences of the above expressions.

\begin{lemma}\label{extra_psi}
	\begin{enumerate}
	 
		\item  
		 For $\lambda <5$, the polynomial $p_\lambda$ has a unique real root $r_\lambda$ and the other two roots are complex conjugates of each other.
	
		\item For $\lambda \in [-1,1)$, in addition to $-\frac{1}{\sqrt{5-\lambda}}$ there are three extraneous fixed points, one is real, we denote it by $\alpha_{\lambda}$ and the other two are complex conjugates of each other.  
		\item For $\lambda =1$, the extraneous fixed point $-\frac{1}{2}$ is multiple with multiplicity two and the other two are complex conjugates of each other. 
			\item For $\lambda \in [-1,1]$, all the three roots of $p_\lambda$ and the poles $\pm i$ are critical points of $C_\lambda$ each with multiplicity two. The other two simple critical points of $C_{\lambda}$ are $\frac{1}{\sqrt{5}}$ and $-\frac{1}{\sqrt{5}}$.
	\end{enumerate}
\label{basic}
\end{lemma}
\begin{proof}
	\begin{enumerate}
		\item  Since $p_\lambda$ is monic and  is of odd degree and also preserves the real line,  $\lim_{x \to \infty}p_{\lambda}(x)=\infty$ and $\lim_{x \to -\infty}p_{\lambda}(x)=-\infty$.  Since $p_\lambda '(x)>0$ for all $x \in \mathbb{R}$, it is strictly increasing and hence it has a unique real root. Clearly the other two roots are complex conjugates of each other as  all the coefficients of $p_\lambda$ are real for all $\lambda <5$.
		\item Putting $a=3$ and $b =\psi(\lambda)$ in Equation~(\ref{iii}) we get that the extraneous fixed points of $C_{\lambda}$ are the solutions of 
		\begin{equation}
		 4z^4+9z^2+\psi(\lambda)z +3=0. 
		 \label{(iii)extraneous}
		\end{equation}  These are nothing but the solutions of  $(z+\frac{1}{\sqrt{5-\lambda}})q(z)=0$ where $q(z)=4z^3-\frac{4z^2}{\sqrt{5-\lambda}}+\frac{(49-9 \lambda )z}{5-\lambda}+ 3 \sqrt{5-\lambda}$. Note that $q'(z)=12z^2-\frac{8z}{\sqrt{5-\lambda} }+\frac{ 49-9 \lambda }{5-\lambda}$ and its discriminant is $-16(27+\frac{8}{5-\lambda}),$ and that is negative for $-1\leq \lambda < 1$. The two roots of $q'$ are non-real. This gives that $q'(z) \neq 0$ and  is either positive or negative for all $z \in \mathbb{R}$, i.e., $q$ is either strictly increasing or strictly decreasing on the real line. Since $q: \mathbb{R} \to \mathbb{R}$,  $\lim_{z \to -\infty}q(z) = -\infty$ and $\lim_{z \to  \infty}q(z) =  \infty$, $q$ is strictly increasing on $\mathbb{R}$   and  hence has  a unique real root. This is the other real extraneous fixed point $\alpha_\lambda$ of $C_{\lambda}$. As all the coefficients of $q$ are real, the other two roots are  complex conjugates of each other which are nothing but the non-real extraneous fixed points of $C_\lambda$ for $\lambda \in [-1,1)$.
		\item For $\lambda =1$, it follows from Equation~(\ref{(iii)extraneous}) that $-\frac{1}{2}$ is a double root of $4z^4 +9z^2+11z+3=0$. The other two roots are found to be  $\beta_1 =\frac{1}{2}( 1+i\sqrt{11}) $ and $\beta_2 =\frac{1}{2}( 1-i\sqrt{11}) $. These are the extraneous fixed points of $C_1$.
	\item Since $p_\lambda$ is generic, its three roots are simple and are superattracting fixed points of $C_\lambda$, each of which is a critical point with multiplicity two. The rest follows from Equation~(\ref{cheby-parameter}) and Equation~(\ref{cheby-prime-parameter}).	\end{enumerate}
\end{proof}
Some estimates are going to be useful.
\begin{lemma}
	Let $-1\leq 
	\lambda \leq 1$. 
  
  \begin{enumerate}
  	\item If $r_\lambda$ is the real root of $p_\lambda$  then $-2 < r_\lambda < -1$ and 
  	$C_\lambda (\frac{1}{\sqrt{5}}) < r_\lambda$. Further, there is $x_0 \in (-\frac{1}{\sqrt{5}},0)$ such that $C_\lambda (x_0)=r_\lambda$.
  	\item If $-1 \leq \lambda <1$ and $\alpha_\lambda$ is the real extraneous fixed point of $C_\lambda$ different from $-\frac{1}{\sqrt{5-\lambda}}$ then   $-\frac{2}{\sqrt{5-\lambda}} < \alpha_\lambda <  -\frac{1}{\sqrt{5-\lambda}}$. For $\lambda =1$, $\alpha_\lambda = -\frac{1}{\sqrt{5-\lambda}}$.
  \end{enumerate} 
	\label{real-extraneous}  
\end{lemma}
\begin{proof}
	For $-1 \leq \lambda \leq 1, 11 \leq \psi(\lambda) \leq  \psi(-1) \approx 11.294$. 

	\begin{enumerate}
		\item 

Note that $p_\lambda (-1)=-4+\psi(\lambda)>0 $ and $p_\lambda (-2)=-14+\psi(\lambda) <0$ for $-1 \leq \lambda \leq 1$. This gives that $-2 < r_\lambda <-1$. For  $-1 \leq  \lambda \leq 1$,  $\psi(\lambda)\geq 11$ and we have $C_\lambda (\frac{1}{\sqrt{5}})=-\frac{\sqrt{5}[25\psi(\lambda)^2+140\sqrt{5}\psi(\lambda)+8]}{1944} <-7.449$. Consequently, we have $C_\lambda (\frac{1}{\sqrt{5}}) < r_\lambda$. 
	
	\par It follows from Equation~(\ref{cheby-parameter}) that $C_\lambda (z)=r_\lambda$ if and only if $S(z)=5z^7 +6z^5-5\psi(\lambda)z^4-3z^3-12 \psi(\lambda) z^2 -\psi(\lambda)^2 z-3 \psi(\lambda)-9 r_\lambda  (z^2 +1)^3=0$. Note that $S (0) =-3( \psi(\lambda)+3 r_\lambda)<0 $ (because $\psi(\lambda)  \geq 11$). Further, $S(-\frac{1}{\sqrt{5}})=  \frac{8}{25 \sqrt{5}}-\frac{1944 r_\lambda}{125}+\frac{\psi(\lambda)}{\sqrt{5}}(\psi(\lambda)-\frac{28}{\sqrt{5}})$.   Note that  $15< \frac{8}{25 \sqrt{5}}-\frac{1944 r_\lambda}{125} < 31$ and the last term  $\frac{\psi(\lambda)}{\sqrt{5}}(\psi(\lambda)-\frac{28}{\sqrt{5}}) $ is   increasing as a function  of $\psi(\lambda) $ (not of $\lambda$) with its minimum value  $\frac{11}{\sqrt{5}}(11-\frac{28}{\sqrt{5}}) \approx -7.487 $. Thus $S(-\frac{1}{\sqrt{5}}) >7$ and we are done by the Intermediate value Theorem.

\item	 Recall from the previous lemma that $q(z) =4z^3-\frac{4z^2}{\sqrt{5-\lambda}}+(9+\frac{4}{5-\lambda})z+ 3 \sqrt{5-\lambda}$ and the real extraneous fixed point of $C_\lambda$ different from $-\frac{1}{\sqrt{5-\lambda}}$ is a root of $q$. As $q(-\frac{2}{\sqrt{5-\lambda}})= -\frac{1}{(5-\lambda )\sqrt{5-\lambda}} (71+12 \lambda -3 \lambda^2)$ and $71+12 \lambda-3 \lambda^2>0$ for all $\lambda  \leq 1$, we have  $q(-\frac{2}{\sqrt{5-\lambda}})<0$. Similarly,  $q(-\frac{1}{\sqrt{5-\lambda}})=  \frac{1}{(5-\lambda )\sqrt{5-\lambda}} (3(6-\lambda)(1-\lambda)) > 0$ for $-1 \leq \lambda < 1$. It is already observed in Lemma \ref{extra_psi}(2) that $q$ is strictly increasing on $\mathbb{R}$. Therefore $-\frac{2}{\sqrt{5-\lambda}} < \alpha_\lambda < -\frac{1}{\sqrt{5-\lambda}}$.
\par For $\lambda =1, \alpha_{\lambda}= -\frac{1}{\sqrt{5-\lambda}}$ (See Lemma~\ref{basic}).
	\end{enumerate} 
\end{proof}
\begin{Remark}
	Let $-1\leq \lambda \leq 1.$ Then
	
	\begin{enumerate}
		\item $p_{\lambda} (\frac{-3}{\sqrt{5-\lambda}})= \frac{3 \lambda^2 -30 \lambda +52}{(5-\lambda)(\sqrt{5-\lambda})}> 0$, $r_\lambda  < \frac{-3}{\sqrt{5-\lambda}} < -\frac{1}{\sqrt{5}}$, and
		\item $r_\lambda < -\frac{3}{\sqrt{5-\lambda}} < -\frac{2}{\sqrt{5-\lambda}}<\alpha_\lambda.$
	\end{enumerate}
\label{Root_of_p}
	\end{Remark}
\begin{figure}[h!]
	\begin{subfigure}{.5\textwidth}
		\centering
		\includegraphics[width=1.1\linewidth]{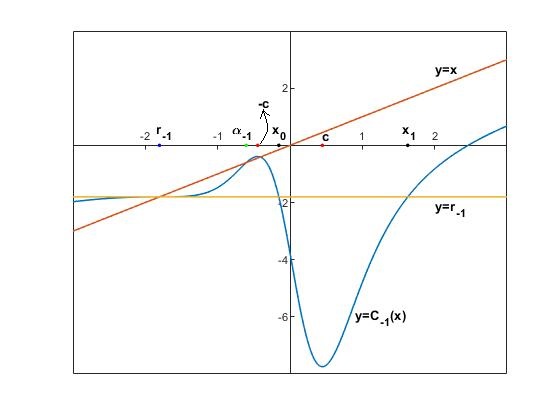}
		\caption{$\lambda = -1$}
	\end{subfigure}%
	\begin{subfigure}{.5\textwidth}
		\centering
		\includegraphics[width=1.1\linewidth]{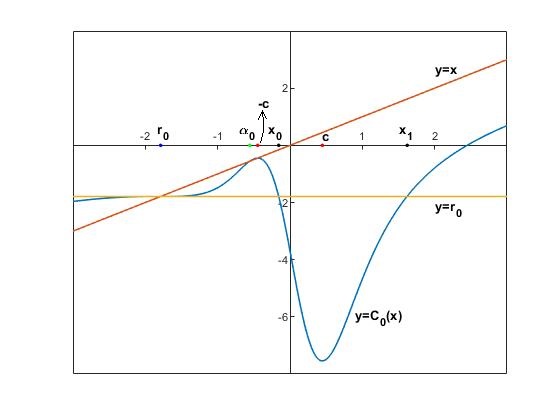}
		\caption{$\lambda =0$ }
		\end{subfigure}\\[1ex]
 \centering
    \begin{subfigure}{0.5\textwidth}
	    \centering
  	    \includegraphics[width=1.2\linewidth]{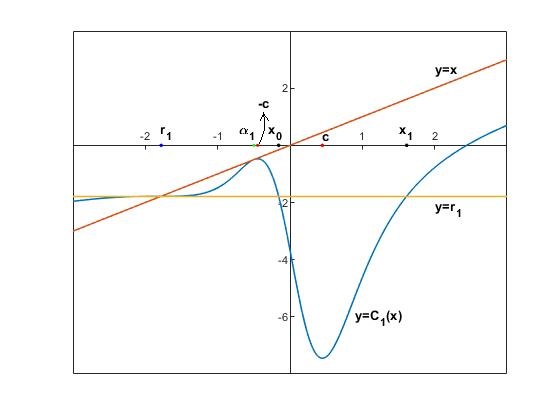}
	    \caption{$\lambda =1$}
    \end{subfigure}
	\caption{The fixed point and the preimages of $r_\lambda$ under $C_\lambda$ }
	\label{Graphs}
\end{figure}
Note that the real root of $p_\lambda$ is a superattracting fixed point of $C_\lambda$. The following lemma describes its immediate basin.
\begin{lemma}
	For $-1 \leq \lambda \leq 1,$ let $r_\lambda$ be the real root of $p_\lambda$. Then it is a super-attracting fixed point of $C_\lambda$ and its  immediate basin $\mathcal{A}_\lambda$  is unbounded. Further, it is simply connected and both the poles of $C_\lambda$ are on its boundary.
	\label{immediatebasin-realfixedpoint}
\end{lemma}
\begin{proof}
Clearly, the root $r_\lambda$ of $p_\lambda$ is simple and is a superattracting fixed point of $C_\lambda$. Note that $C_\lambda$ is strictly increasing in $(-\infty, -\frac{1}{\sqrt{5}})$, strictly decreasing in $(-\frac{1}{\sqrt{5}},  \frac{1}{\sqrt{5}})$ and strictly increasing thereafter.
Let $c=\frac{1}{\sqrt{5}}$.
Since $r_\lambda < -c$ (by Remark \ref{Root_of_p}(1)), $C_\lambda$ is strictly increasing in $(-\infty, r_\lambda]$. Further, by the preceeding remark, the only other real extraneous fixed point of $C_\lambda$ is greater than $r_\lambda$. Therefore,  for all $x \in (-\infty,r_\lambda)$, either $C_\lambda(x)<x$ or $C_\lambda(x)>x$. The first possibility leads to a strictly decreasing sequence $\{C_{\lambda}^n (x) \}_{n>0}$ which must converge to $-\infty$. But this is not possible as $\infty$ is a repelling fixed point of $C_\lambda$. Therefore  $C_\lambda (x)>x$ and  $C_\lambda ^n(x) \to r_\lambda$ as $n \to \infty$ for all $x \in (-\infty, r_\lambda] $ giving that  $(-\infty, r_\lambda] \subset \mathcal{A}_\lambda$.  In other words,  $\mathcal{A}_\lambda$  is unbounded.
\par It follows from Lemma~\ref{real-extraneous}(1) that the critical value $C_\lambda(c) \in \mathcal{A}_\lambda$.  As the extraneous fixed point $-\frac{1}{\sqrt{5-\lambda}}$ is either attracting or parabolic, its basin (attracting or parabolic) contains a critical point. But, the only available critical point  is $-c$. Therefore $-c$ is in the immediate basin of attraction or immediate parabolic basin of  $-\frac{1}{\sqrt{5-\lambda}}$. So $\mathcal{A}_\lambda$ contains at most one critical point other than $r_\lambda$ and that can  be $c$ only. We are going to show that this is not the case, i.e., $c \notin \mathcal{A}_\lambda$.
\par
Suppose on the contrary that $c \in \mathcal{A}_\lambda$. Then 
$C_\lambda ([x_0, c]) = [C_\lambda (c ), r_\lambda]$ and it gives that  $C_\lambda ([x_0, c]) \subset \mathcal{A}_\lambda$. Note that there is a real pre-image of $r_\lambda$ in $(c,\infty)$. Let it be $x_1$. Then  $C_\lambda$ maps $[c, x_1]$   onto the same interval $[C_\lambda (c ), r_\lambda]$ giving that $x_1 \in A_{\lambda}$ whenever $c \in A_\lambda$. The locations of $x_0, x_1, r_\lambda$ and the critical points  are shown in Figure~\ref{Graphs} for $\lambda=-1,0$ and $1$, where the red dots represent the critical points.
\par 
The Bottcher coordinate   $\phi$ is locally defined and univalent at $r_\lambda$ (See Theorem 9.3, \cite{Milnor}), i.e., there is a simply connected domain $\tilde{U} \subseteq \mathcal{A}_{\lambda}$  such that $\phi: \tilde{U} \to \phi(\tilde{U}) \subseteq \{z: |z|<1\}$ is conformal. Since the critical point $c$ is assumed to be in $\mathcal{A}_\lambda, \phi$ cannot be extended conformally to the whole of $\mathcal{A}_\lambda$. In other words, there is a maximal $r \in (0,1)$ such that $\phi^{-1}$ is well-defined on $D_r =\{z: |z|<r\}$. Let $U=\phi^{-1}(D_r)$. Clearly $c \in \partial U$. Let $V=C_\lambda (U)$. Then $\overline{V} \subset U$  because $\phi(U)=D_r$ and $\phi \circ C_\lambda \circ \phi^{-1}(z)=z^3$ on $D_r$ (by Theorem 9.3,~\cite{Milnor}), and $\overline{D_{r^3}} \subset D_r$. It also follows that $C_\lambda: U \to V$ is a proper map of degree three. 
Since $\phi^{-1}$ is well-defined and conformal on $D_r \supsetneq D_{r^3}$,  $\phi^{-1}(D_{r^3})=V$ is a Jordan domain. Let $\gamma =\partial V \setminus \{C_\lambda (c)\}$. As the local degree of $C_\lambda$ at $c$ is two, there are two branches of $C_{\lambda}^{-1}$ and each is well-defined on $\gamma$ by the Monodromy theorem. Since there is no critical value of $C_\lambda$ on $\gamma$, the images of $\gamma$ under each of these branches are Jordan arcs. Let these images be $\sigma$ and $\sigma'$. Then $\sigma \cap \sigma' =\emptyset$ and each of $\overline{\sigma}$ and $\overline{\sigma'}$ is a Jordan curve  with  $\overline{\sigma} \cap \overline{\sigma'}=\{c\}$. In fact, the bounded components  of $\widehat{\mathbb{C}} \setminus \overline{\sigma}$ and $\widehat{\mathbb{C}} \setminus \overline{\sigma'}$ are the images of $V$ under the two branches of $C_{\lambda}^{-1}$. This is because the unbounded components of $\widehat{\mathbb{C}} \setminus \overline{\sigma}$ and $\widehat{\mathbb{C}} \setminus \overline{\sigma'}$  contains a point of the Julia set of $C_\lambda$, namely $\infty$ and therefore no such unbounded component can be mapped into $V$, which is in the Fatou set of $C_\lambda$.  Clearly, these complementary bounded components are disjoint. One of these must be $U$.  Assume without loss of generality that $U$ is the bounded component of $\widehat{\mathbb{C}} \setminus \overline{\sigma}$. The possible figures of $U$ and $U'$ is given in the lefthand side image of  Figure~\ref{U_pic}.

\begin{figure}[h!]
 
		\centering
		\includegraphics[width=0.46\linewidth]{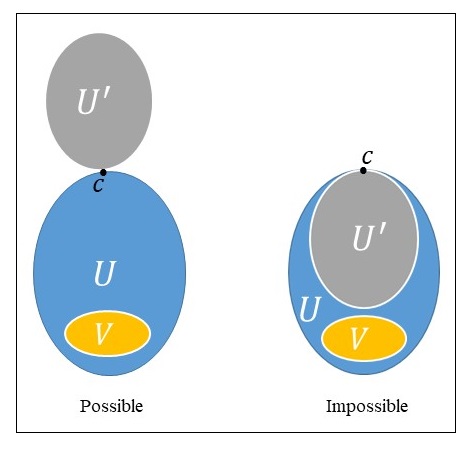}
		\caption{The possible position of $U$ and $U'$.}
		 \label{U_pic}
\end{figure} 
 Let $U'$ be the bounded component of $\widehat{\mathbb{C}} \setminus \overline{\sigma'}$.
%
Now $r_\lambda \in U$ and $U'$ contains a pre-image, say $x^*$ of $r_\lambda$ such that $x^* \neq r_{\lambda}$.

 Now consider a simply connected open set $W_0$ containing the closure of $U \cup U'$ and let $W_1$ be the component of $C_{\lambda}^{-1}(W_0)$ containing $r_\lambda$. Then $C_\lambda : W_1 \to W_0$ is a proper map with some  degree $d$. Clearly $d \geq 4$ as there are at least four pre-images of $r_{\lambda}$ in $W_1$ counting multiplicity, namely $r_\lambda$ itself with multiplicity $3$ and $x^*$ with multiplicity $1$. Since  $C_\lambda$ is a proper map (of  degree $7$) from $\widehat{\mathbb{C}}$ onto itself and  $W_1 \neq \widehat{\mathbb{C}}$, the connectivity of $W_1$ is non-zero and finite. In fact, each component of $\widehat{\mathbb{C}} \setminus W_1$ is mapped onto $\widehat{\mathbb{C}} \setminus W_0$ and there cannot be more than seven such components. Since $W_1$ contains two critical points, namely $r_\lambda$ with multiplicity $2$ and $c$ with multiplicity $1$, it follows from the Riemann-Hurwitz formula (Lemma~\ref{r-h}) that  $c(W_1)-2= d(c(W_0)-2)+3=3-d$. This  gives that the connectivity of $W_1 $ is less than or equal to $ 1$. Thus $W_1$ is simply connected and $d=4$. Note that $W_1$ contains only one pre-image of $r_\lambda$ different from itself and this must be $x^*$. Applying this argument again we get that $C_\lambda: W_2 \to W_1$ is a proper map of degree $4$, $W_2$ is simply connected and $x^*$  is the only pre-image of $r_\lambda$ different from itself, belonging to $W_2$ where $W_2$ is the component of $C_{\lambda}^{-1}(W_1)$ containing $r_\lambda$. It follows by induction that for each $n \geq 1$, if  $W_{n}$ is the component of $C_{\lambda}^{-1}(W_{n-1})$ containing $r_\lambda$ then $C_\lambda: W_{n} \to W_{n-1}$ is a proper map of degree $4$, $W_{n}$ is simply connected and $x^*$ is the only pre-image of $r_\lambda$, different from $r_\lambda$ belonging to  $W_{n}$.
 
 \par Since $x_0, x_1 \in \mathcal{A}_\lambda$, one of them, say $x_1$ is different from $x^*$. Thus $x_1 \notin W_{n}$ for any $n$. Consider an arc in $\mathcal{A}_\lambda$ joining $r_\lambda$ with $x_1$.  This arc cannot be contained in $W_n$ and  intersects its boundary for each $n$. Let $w_n$ be a point of such intersection. Then $w_n$ has an accumulation point, say $w$ in $\mathcal{A}_\lambda$. Considering a sufficiently small neighborhood $N_w$ of $w$ contained in $\mathcal{A}_\lambda$ we observe that $C_\lambda ^n (N_w)$ intersects the boundary of $W_0$ for all sufficiently large $n$. However, there is an $n_0$ such that $C_\lambda ^n (N_w) \subset V \subsetneq \overline{V} \subset W_0$ for all $n >n_0$.   This is a contradiction and we prove that $c \notin A_\lambda$.
 \par  It now follows from a well-known result (Theorem 9.3, \cite{Milnor}) that $\mathcal{A}_\lambda$ is simply connected. 
\par 
By Lemma~\ref{poleinboundary}, there is a pole of $C_\lambda $ on the boundary of $\mathcal{A}_{\lambda}$. The Fatou component $\mathcal{A}_{\lambda} $ is symmetric about the real line by Remark~\ref{5-6}~(\ref{lambda<5}). Since the poles are complex conjugates of each other, the other pole is also on the boundary of $\mathcal{A}_\lambda$. 

\end{proof}

We   now  present the proof of  Theorem \ref{connected_J_set}.
\begin{proof}[Proof of Theorem \ref{connected_J_set}]
Note that $C_\lambda$ has a fixed point at $\infty$ and that is repelling. It follows from Lemma~\ref{immediatebasin-realfixedpoint} that the immediate basin $\mathcal{A}_{\lambda}$ of the real superattracting fixed point of $C_\lambda$ corresponding to the real root of $p_\lambda$  is unbounded, simply connected and contains both the poles of $C_\lambda$ on its boundary.  Now it follows from Lemma~\ref{poleinboundary} that  the Julia set of $C_\lambda$ is connected.
	
 \end{proof}
  
\section{Concluding remarks}
For $C_{\lambda}, \lambda \in [-1,1]$, all the critical points except the  poles are in the attracting or parabolic basins. In fact, the Fatou set of $C_\lambda$ is the union of the   basins of the three superattracting fixed points corresponding to the three roots of $p_\lambda$ and the basin of the extraneous fixed point (which is parabolic for $\lambda =\pm 1$ and attracting otherwise). In particular, the Fatou set of $C_\lambda$ does not contain any Siegel disk or any Herman ring.

It is observed from the graph of $\psi(\lambda)$ (Figure~\ref{psi-lambda}) that for each $\lambda \in [-1,1)$ there is a $\lambda' \in (1, \delta]$ where $\delta$ is the positive number  satisfying  $\psi(\delta)=\psi(-1)$ such that $\psi(\lambda)=\psi(\lambda')$. This gives that $p_{\lambda}=p_{\lambda'}$ and consequently, $C_{\lambda}=C_{\lambda'}$. Thus Theorem~\ref{connected} is true for all $\lambda \in [-1,\delta]$.
In terms of the real extraneous fixed points it means  the following. For $\lambda \in [-1,1)$, the multiplier of $-\frac{1}{\sqrt{5-\lambda}}$ is $\lambda$ whereas the multiplier of the second real extraneous fixed point $\alpha_\lambda$ is in $(1, \delta)$ and hence is repelling. For $\lambda \in (1, \delta)$, the extraneous fixed point  $-\frac{1}{\sqrt{5-\lambda}}$ becomes repelling making $\alpha_\lambda $ attracting.
\par 
 For $\lambda \in [-1,1]$, the forward orbits of the critical points, $\pm \frac{1}{\sqrt{5}}$ remains in the real line. This along with Lemma~\ref{basic-dynamics-lemma} gives that the  two non-real extraneous fixed points cannot be attracting or parabolic. These are in fact, repelling. To see it, recall that each extraneous fixed point other than $-\frac{1}{\sqrt{5-\lambda}}$ is a solution  of
 $q(z)=4z^3-\frac{4}{\sqrt{5-\lambda}}z^2+\frac{49-9\lambda}{5-\lambda}z+3\sqrt{5-\lambda}=0.$ 
Among them one, namely  $\alpha_\lambda$ is already known to be real and other two say, $ \zeta $ and $ \bar{\zeta} $ are complex conjugates of each other. Since $q(z)=4 (z-\alpha_\lambda)(z-\zeta)(z-\overline{\zeta})$, comparing the constant terms we get,
  $ \alpha_\lambda \zeta \bar{\zeta}=-\frac{3\sqrt{5-\lambda}}{4}. $ In other word,  $|\zeta|^2=-\frac{3\sqrt{5-\lambda}}{4\alpha_\lambda}$.  Since $-\frac{2}{\sqrt{5-\lambda}}<\alpha_\lambda<-\frac{1}{\sqrt{5-\lambda}}$ (by Lemma~\ref{real-extraneous} (2) ), we have $\frac{3(5-\lambda)}{8}<-\frac{3\sqrt{5-\lambda}}{4\alpha_\lambda}<\frac{3(5-\lambda)}{4}$. Consequently,  $\frac{3}{2}<|\zeta|^2<\frac{9}{2}$ and  $ \frac{13}{3}<5-\frac{1}{|\zeta|^2}<\frac{43}{9} $. Recall that $|\lambda_{\bar{\zeta}}|=|\lambda_\zeta|$ and this is $|5-\frac{1}{\zeta^2}|\geq 5-\frac{1}{|\zeta|^2}>\frac{13}{3}$. 
      
We conclude by presenting several problems for further investigation.
\begin{enumerate}
	\item The Julia set of the Chebyshev's method  applied to a cubic polynomial with an attracting extraneous fixed point (with non-real multiplier) may be connected. But the arguments used in this article seems to be insufficient to verify it. 
	\item  The images in Figure~\ref{Cheby-$-1$}, Figure~\ref{Cheby-$0$} and Figure~\ref{Cheby-$1$}  suggest that the attracting/parabolic domain  corresponding to the  extraneous attracting/parabolic fixed point is bounded. However, this is yet to be proved.
	
	\item We believe that all the immediate basins of attractions of the superattracting fixed points corresponding to the roots of $p_\lambda$ are unbounded for $\lambda \in [-1,1]$. This article  proves it only for the real root of  $p_\lambda$.
		\item  A rational map is called geometrically finite if the postcritical set $\overline{\mathcal{P}_R} \cap \mathcal{J}(R)$ is finite where $\mathcal{P}_R$ is the union of all the forward orbits of all the critical points of $R$.	It follows from the proof of Theorem \ref{connected_J_set} that  $C_\lambda$ is geometrically finite for $\lambda \in [-1,1]$. It is also clear from Proposition~\ref{unicritical-nongeneric} that $C_p$ is geometrically finite for all cubic unicritical and non-generic polynomials.  Since the Julia set is connected in each of these cases, it is locally connected by ~\cite{lei-96}. The nature of the boundaries of the Fatou components can be explored.
	\item For $\lambda \in (-\infty, -1) \cup (\delta,5)$, $\psi(\lambda)$ is a real number and  $p_\lambda$ preserves the real line, and it has a unique real root. The dynamics of $C_\lambda$ is symmetric about the real line by Remark~\ref{5-6}(\ref{lambda<5}). The forward orbits of the two real critical points of $C_\lambda$ remain in the real line. It  seems plausible to analyze these forward orbits and determine the dynamics of $C_\lambda$.
\end{enumerate} 
 \section*{Acknowledgement}
 The second author is supported by the University Grants Commission, Govt. of India.

\end{document}